\documentclass[11pt]{amsart}

\usepackage{amsmath} 
\usepackage{amssymb} 
\usepackage{amscd} 
\usepackage{tabularx} 
\usepackage[all,cmtip,line]{xy}
\usepackage{enumerate}
\usepackage{hyperref}
\usepackage{bm}
\usepackage{enumitem}
\usepackage{amsthm}
\newtheorem{thm}{\textbf{Theorem}}[section]
\newtheorem{defn}[thm]{\textbf{Definition}}
\newtheorem{prop}[thm]{\textbf{Proposition}}
\newtheorem{lem}[thm]{\textbf{Lemma}}

\newtheorem{rem}[thm]{\textbf{Remark}}

\newtheorem*{thmn}{\textbf{Theorem}}

\def\Q{\mathbb{Q}}
\def\Z{\mathbb{Z}}
\def\C{\mathbb{C}}
\def\A{\mathbb{A}}
\def\R{\mathbb{R}}
\def\Gal{\operatorname{Gal}}

\def\GL{\operatorname{GL}}
\def\SL{\operatorname{SL}}
\def\SO{\operatorname{SO}}

\def\til{\widetilde}

\def\Qbar{\overline{\mathbb{Q}}}
\def\rhobar{\overline{\rho}}
\def\xbar{\overline{x}}
\def\zbar{\overline{z}}

\newcommand{\Hom}{\operatorname{Hom}}
\newcommand{\Spec}{\operatorname{Spec}} 
\newcommand{\Sp}{\operatorname{Sp}} 
 
\newcommand{\Sym}{\operatorname{Sym}} 
\newcommand{\cO}{\mathcal{O}}
\newcommand{\cH}{\mathcal{H}}
\newcommand{\cM}{\mathcal{M}}
\newcommand{\ram}{\mathrm{ram}}
\newcommand{\ur}{\mathrm{ur}}
\newcommand{\Ind}{\operatorname{Ind}}
\newcommand{\dslash}{\mathbin{/\mkern-4mu/}}

\title{On endoscopic \MakeLowercase{\textit{p}}-adic automorphic forms for $\SL_2$}
\author{Judith Ludwig}
\email{jludwig@math.uni-bonn.de}

\begin{document}
\date{\today}
\maketitle
\begin{abstract} We show the existence of some non-classical cohomological $p$-adic automorphic eigenforms for $\SL_2$ using endoscopy and the geometry of eigenvarieties. These forms seem to account for some non-automorphic members of classical global $L$-packets.  
\end{abstract}

\section{Introduction}

Let $\Pi(\theta)$ be an endoscopic $L$-packet of $\SL_2/\Q$ associated to an algebraic character 
\[\theta: \mathrm{Res}_{F/\Q}\mathbb{G}_m (\A) \rightarrow \C^*,\]
where $F/\Q$ is an imaginary quadratic extension in which $p$ splits. For a sufficiently small tame level $K^p \subset \SL_2(\A_f^p)$ an automorphic representation $\pi \in \Pi(\theta)$ gives rise to a point $x$ of critical slope on an eigenvariety of tame level $K^p$. 

Let $E(K^p,\mathfrak{m})$ be such an the eigenvariety, built from the suitably localized completed cohomology groups $\til{H}^1(K^p)_\mathfrak{m}$ as in \cite{emerton}. Emerton's eigenvariety comes equipped with a coherent module $\cM$ of $p$-adic automorphic forms of tame level $K^p$. If $z$ is a point on $E(K^p,\mathfrak{m})$ coming from a classical automorphic representation, the fibre $\cM_{\zbar}:= \cM_z \otimes k(z)$ contains a subspace 
\[\cM^{cl}_{\zbar} \subset \cM_{\zbar}
\]
of classical automorphic forms.  
Our main theorem concerns the difference between this classical subspace and the whole fibre at critical points $x$ coming from endoscopic $L$-packets as above. 
Under some technical assumptions on the \mbox{$L$-packet} (see \textit{Hypothesis $(\star)$} in Section \ref{sec: mainthm}) we can prove that these spaces differ. The main result of this paper is the following.  
 
\begin{thmn}[Theorem \ref{mainthm}] Let $\Pi(\theta)$ be an endoscopic $L$-packet satisfying \textit{Hypothesis~$(\star)$}. Then there exists a tame level $K^p$ such that the fibre $\mathcal{M}_{\xbar}$ at the critical point $x \in E(K^p,\mathfrak{m})(\Qbar_p)$ contains non-classical $p$-adic automorphic forms, i.e.
\[\mathcal{M}_{\xbar}/\mathcal{M}_{\xbar}^{cl}\neq 0.
\]
\end{thmn}

Our motivation for this theorem is twofold. Firstly, we would like to understand how endoscopy works in the setting of the $p$-adic Langlands programme. As a first step we would like to know whether \textit{non automorphic} members of classical global $L$-or $A$-packets play a role in the $p$-adic setting. We ask the following vague question: 
Let $G/\Q$ be a symplectic, orthogonal or unitary group. Let $\Pi$ be a global endoscopic cohomological $L$- or $A$-packet and let $\pi \in \Pi$ be an element such that the multiplicity $m(\pi)$ in the automorphic spectrum of $G$ is equal to zero.
Then does $\pi$ occur $p$-adically? Is there a point on a suitable eigenvariety and a non-classical $p$-adic eigenform $f$ for the system of Hecke eigenvalues determined by $\pi$ such that $f$ accounts for $\pi$? The non-classical forms appearing in our theorem seem to provide a $p$-adic account for non-automorphic representations in $\Pi(\theta)$, although we caution the reader that this is to be understood only in a vague sense.

The second motivation to study the spaces $\cM_{\xbar}$ of the theorem comes from a related result for the group $\GL_2$. 
There is a strong relationship between automorphic representations of the groups $\GL_2$ and $\SL_2$; every member of a global $L$-packet of $\SL_2$ occurs in the restriction to $\SL_2$ of an automorphic representation of $\GL_2$. In the language of modular forms, endoscopic cohomological $L$-packets are those arising from modular forms with complex multiplication by an imaginary quadratic field. Now consider points of critical slope on a Coleman-Mazur eigencurve that come from such modular forms. These points are interesting from a geometric point of view as the map down to weight space ramifies there. Equivalently but more importantly for us, the corresponding generalized eigenspace of overconvergent forms contains non-classical forms (see Remark \ref{CMB}). Our theorem can be seen as an analogue for $\SL_2$ of this fact, with the notable difference that our theorem concerns honest eigenspaces. 

Although the situations for $\GL_2$ and $\SL_2$ are clearly related, we cannot deduce our result from the $\GL_2$-situation. 
Instead we use endoscopy and geometry on the $\SL_2$-eigencurve to prove our result. We briefly explain the strategy and the key ideas of the proof. 

\textbf{Strategy}: Show that a point $x$ as in the theorem has an open affinoid neighbourhood $U \subset E(K^p,\mathfrak{m})$ that contains a Zariski-dense set of classical \textit{stable} points, i.e., points arising from $L$-packets $\Pi$ that are stable. 
Then compare the dimension of $\cM^{cl}_{\xbar}$ to the dimension of the classical subspace $\cM_{\zbar}^{cl}$ of the fibre at a stable point $z$ and show that the rank at stable points is larger. This implies the theorem as the fibre rank of a coherent sheaf is semi-continuous. 
  
We briefly explain why there is more contribution from automorphic representations at stable points than at endoscopic points. For that first note that the classical subspace $\cM_{\zbar}^{cl}$ of the fibre $\cM_{\zbar}$ at any classical point $z$ of weight $k$ decomposes into a direct sum
\[\cM_{\zbar}^{cl} \cong \bigoplus_{\pi \in X(z)} m(\pi)\  M_{\pi} \otimes H^1_{rel.Lie}(\mathfrak{g},\SO_2(\R), \pi_\infty \otimes \Sym^k(\C^2))\]
indexed over a certain subset $X(z)$ of the global $L$-packet $\Pi(z)$ defined by the point~$z$. The number $m(\pi)$ is the multiplicity of $\pi$ in the cuspidal automorphic spectrum and $M_{\pi}$ is the tensor product of a subspace of the smooth Jacquet module $J_B(\pi_p)$ and the spaces $\pi_l^{K_l}$, with $l$ running through the set of primes where $K_l\neq \SL_2(\Z_l)$. 

\textbf{Crucial observation}: Ignoring the spaces $M_\pi$, the classical subspace of the fibre at $x$ has half as many non-zero summands as the space at a stable point $z$. The reason for this is that in a stable packet all elements are automorphic, contrary to the situation for the $L$-packet $\Pi(\theta)$. 
The archimedean $L$-packet $\Pi(\theta)_\infty = \{D_{k(\theta)}^+,D_{k(\theta)}^-\}$ is a discrete series $L$-packet of size two and given $\tau_f \in \Pi(\theta)_f$ exactly one of the representations 
\[\tau_f\otimes D_{k(\theta)}^+, \tau_f \otimes D_{k(\theta)}^-
 \]
is automorphic.
So if we manage to control the contributions of the terms $M_{\pi}$ in a family, in other words if we can find some tame level that minimizes the terms $M_\pi$ for $\pi \in X(x)\subset \Pi(\theta)$, then we get the estimate that we seek. This is achieved using on the one hand the theory of newforms and conductors for $\SL_2$ as developed in \cite{LR}. The main technical result here is Proposition \ref{fibres}. The second ingredient is a rigidity result on the behaviour of $L$-packets in families (see Proposition \ref{rigidity}).

In \cite{p-adicpackets}, we constructed non-classical points on certain eigenvarieties for a definite inner form $G$ of $\SL_2$ using endoscopy and a $p$-adic version of the Labesse-Langlands transfer \cite{p-adicLL}. We used a local $L$-packet of size two to change multiplicities in a global $L$-packet, however in \cite{p-adicpackets} that auxiliary prime was non-archimedean. In this paper we explore the fact that for the split group $\SL_2(\R)$ the discrete series $L$-packets are of size two.
The advantage of the method of this paper is that it should work for other groups as well. 

\textbf{Outline of the paper.} After recalling some background on the cohomology of the symmetric space of $\SL_2$ in Section 2, we summarize results of \cite{LR} on conductors of smooth representations of $\SL_2(\Q_l)$. The important results are the dimension formulas of Proposition \ref{newforms}. Section 4 is on the local symmetric square lifting, we need Proposition \ref{sym2} to prove the rigidity result \ref{rigidity} in the next section. The rest of Section 5 sets up the eigenvarieties we use. We then prove our main theorem in the last section.

\textbf{Acknowledgements}. The author would like to thank Kevin Buzzard, Ga{\"e}tan Chenevier, Eugen Hellmann and Yichao Tian for helpful conversations, Joachim Schwermer for pointing her to the reference \cite{schwermer}, and James Newton and Peter Scholze for their helpful comments on an earlier draft of this manuscript. The author was supported by the SFB/TR 45 of the DFG.

\section{Background}\label{sec: background}

Let $\mathfrak{g}=\mathfrak{sl}_2$ be the Lie-algebra of $\SL_2(\R)$ and define $K_\infty = \SO_2(\R)$. Let $W_\C:=\operatorname{Sym}^{k}(\C^2)\cong\breve{W}_\C$ be the algebraic representation of $\SL_2$ of highest weight $k$.

Below we need the following result on the relative Lie-algebra cohomology.
\begin{lem}[{\cite[Section 2.1]{Labesse-Schwermer}}]\label{(g,K)}
\[
H^1_{rel.Lie}(\mathfrak{g}, K_\infty, W_\C \otimes \pi_\infty) = \begin{cases}
      \C, & \text{if}\ \pi_\infty = D^{\pm}_{k+1} \\
      0, & \text{otherwise.}
    \end{cases}
\]
\end{lem}
Here $D_{k+1}^+$ and $D_{k+1}^-$ denote the holomorphic and antiholomorphic discrete series representation of weight $k+1$, i.e., the infinite-dimensional constituents of the non-unitary principal series representation $I_{k+1}$ of $\SL_2(\R)$ associated to the character $t \mapsto t^{k+1} (\mathrm{sgn}(t))^{k}$ of $\R^*$, c.f.\ \cite{Labesse-Schwermer} Section 2.1.  

For a compact open subgroup $K=\prod_l K_l \subset \SL_2(\A_f)$ consider the symmetric space
\[Y(K)= \SL_2(\Q)\backslash \SL_2(\A)/K \SO_2(\R).
\]
Let $\mathcal{W}_k$ denote the local system on $Y(K)$ attached to $W_\C$ and consider the singular cohomology group
\[H^1(K,W_\C):= H^1(Y(K),\mathcal{W}_k).\]
Recall that the parabolic cohomology\footnote{This is often called interior cohomology.} 
\[H_{\mathrm{par}}^1(K,W_\C):=H_{\mathrm{par}}^1(Y(K),\mathcal{W}_k)
\] 
is defined as the image of the natural map 
\[H^1_c(Y(K),\mathcal{W}_k) \rightarrow H^1(Y(K),\mathcal{W}_k).
\]

By \cite[Corollary 2.3]{schwermer} parabolic cohomology and cuspidal cohomology agree for the groups $\SL_2/\Q$, we have a decomposition 
\begin{equation}\label{par} H_{\mathrm{par}}^1(K,W_\C)= \bigoplus_{\substack{\pi \ adm.\ rep.\\ of \SL_2(\A)}} m(\pi) \ \pi_f^{K}\otimes H^1_{rel.Lie}(\mathfrak{g}, K_\infty, W_\C \otimes \pi_\infty), 
\end{equation}
where $m(\pi)$ is the multiplicity of $\pi$ in the cuspidal automorphic spectrum. 

Let $S(K)$ be a finite set of primes such that $K_l = \SL_2(\Z_l)$ for all $l \notin S(K)$. The unramified Hecke algebra
\[\mathcal{H}^{\ur} = \bigotimes{'}_{l \notin S(K)}  C^{\infty}_c(\SL_2(\Q_l)\dslash\SL_2(\Z_l))
\] 
acts on $H^1(K,W_\C)$ and $H_{\mathrm{par}}^1(K,W_\C)$. We recall that for a \textit{cuspidal} automorphic representation $\pi$ of $\SL_2(\A)$ occurring in $\varinjlim_{K'} H^1(K',W_\C)$ with $\pi_f^K \neq 0$ and unramified system of Hecke eigenvalues $\lambda:\mathcal{H}^{\ur}\rightarrow \C^*$, we have an isomorphism
\[H^1(K,W_\C)^\lambda \cong H_{\mathrm{par}}^1(K,W_\C)^\lambda. 
\]

We also remind the reader that any cuspidal automorphic representation $\pi$ of $\SL_2$ occurs with multiplicity one in the space of automorphic forms (\cite[Theorem 4.1.1]{Ramakrishnan}).

For an irreducible smooth representation $\pi_p$ of $\SL_2(\Q_p)$ let $J_B(\pi_p)$ be its Jacquet-module with respect to the upper triangular Borel $B$. It is a smooth representation of the torus $T\subset \SL_2(\Q_p)$ of diagonal matrices. For a character $\chi:T \rightarrow \C^*$ we consider its eigenspace $J_B(\pi_p)^{\chi}$. 
\begin{lem}\label{jacdim}
Let $\pi_p \cong \Ind_B^{\SL_2(\Q_p)}(\chi)$ be an irreducible unramified principal series representation. Then $J_B(\pi_p)^{\chi}$ is one-dimensional. 
\end{lem}
\begin{proof} The principal series representation is irreducible, therefore $\chi$ is not a quadratic character, i.e. $\chi \neq \chi^{-1}$. The Jacquet-module $J_B(\Ind_B^{\SL_2(\Q_p)}(\chi))$ is two-dimensional and its semi-simplification as a $\C[T]$-module is given by 
\[J_B(\Ind_B^{\SL_2(\Q_p)}(\chi))^{ss}\cong \delta_B^{1/2}\chi \oplus \delta_B^{1/2}\chi^{-1},\]
where $\delta_B$ is the modulus character.
\end{proof}

\section{Newforms for $\SL_2(\Q_l)$} \label{sec: newforms}
In Section \ref{sec: mainthm} we need some of the results of \cite{LR} which we summarize here.  
Let $l\neq 2$ be a prime number. Given an irreducible smooth representation $\til{\pi}$ of $\GL_2(\Q_l)$ its restriction to $\SL_2(\Q_l)$ breaks up into a finite direct sum of irreducible smooth representations of $\SL_2(\Q_l)$, each occurring with multiplicity one. A local $L$-packet of $\SL_2(\Q_l)$ is a finite set of irreducible smooth representations of $\SL_2(\Q_l)$ arising in this way. They are of size $1,2$ or $4$. We refer to \cite{LL} for details. 
 
\begin{defn} 
\begin{enumerate}
	\item We say a local $L$-packet of $\SL_2(\Q_l)$ is unramified, if it contains a representation $\pi_l$ with $\pi_l^{\SL_2(\Z_l)} \neq 0$. 
Otherwise we call it ramified.
\item A local $L$-packet $\Pi$ of $\SL_2(\Q_l)$ is called supercuspidal if one (equiv. all) $\pi \in \Pi$ are supercuspidal.
\end{enumerate}
\end{defn}

\begin{rem}\label{sc2} In an unramified local $L$-packet the element $\pi_l$ with $\pi_l^{\SL_2(\Z_l)} \neq 0$ is unique and we denote it by $\pi_l^0$. 
Supercuspidal local $L$-packets are of cardinality two or four. As we have assumed $l \neq 2$, all supercuspidal representations of $\GL_2(\Q_l)$ can be constructed from  characters $\chi:E^*\rightarrow \C^*$ of quadratic extensions $E/\Q_l$ which do not factor through the norm $\mathrm{N}^E_{\Q_l}$. We write $\pi(\chi)$ for this representation and refer to Chapter 39 of \cite{BH} for the construction. Let $\overline{\chi}$ denote the conjugate of $\chi$ under the non-trivial element of $\tau \in \Gal(E/\Q_l)$, so $\overline{\chi}(x)=\chi(\tau{x})$. The $L$-packet $\Pi(\pi(\chi))$ defined by $\pi(\chi)$ is of size two unless $\chi\overline{\chi}^{-1}$ is a quadratic character, when it is of size four. 
\end{rem}

For $m\geq0 $ consider the compact open subgroups of $\SL_2(\Z_l)$ 
\[ K_1(m):=\left\{ \left(\begin{array}{cc} a & b\\ c & d \end{array}\right) \in \SL_2(\Z_l) | c \equiv 0 , d \equiv 1 \mod l^m  \right\}
\]
\[\text{and } K_0(m):=\left\{ \left(\begin{array}{cc} a & b\\ c & d \end{array}\right) \in \SL_2(\Z_l) | c \equiv 0 \mod l^m  \right\}.
\]
The group $K_1(m)$ is normal in $K_0(m)$. The quotient $K_0(m)/K_1(m)$ is isomorphic to $\Z_l^*/ (1+l^m \Z_l)$. 
Define $\alpha:=\left(\begin{array}{cc} l & \\  & 1 \end{array}\right) \in \GL_2(\Q_l)$ and let 
\[K'_0(m):= \alpha^{-1}K_0(m) \alpha \subset \SL_2(\Q_l)
\]  
be the conjugate subgroup. 

Let $\eta:\Z_l^*\rightarrow \C^*$ be a character and let $c(\eta)$ be its conductor. Then for any $m\geq c(\eta)$, $\eta$ defines a character of $K_0(m)$ via 
$ \left(\begin{array}{cc} a & b\\ c & d \end{array}\right)\mapsto \eta(d)$, which we still denote by $\eta$ and similarly for $K'_0(m)$.
Let $\pi$ be an irreducible smooth representation of $\SL_2(\Q_l)$. Define
\[ 
\pi_\eta^{K_0(m)}:= \left\{v \in \pi: \pi\left(\left(\begin{array}{cc} a & b\\ c & d \end{array}\right)\right) v = \eta(d) v  \ \forall \left(\begin{array}{cc} a & b\\ c & d \end{array}\right) \in K_0(m) \right\}
\]
and $\pi_\eta^{K'_0(m)}$ analogously. 

Let $\omega_\pi$ denote the central character of $\pi$. Note that if  $\eta(-1) \neq \omega_\pi(-1)$, then $ \pi_\eta^{K_0(m)}=0$ and $\pi_\eta^{K'_0(m)}=0$. 
We remark here that the central characters of any two elements in a local $L$-packet agree. This is obvious from the definition of a local $L$-packet as the set of  representations occurring in the restriction of an irreducible smooth representation of $\GL_2(\Q_l)$. 

The authors of \cite{LR} define \textit{the conductor} $c(\pi)$ of an irreducible smooth representation $\pi$ of $\SL_2(\Q_l)$ as
\[c(\pi):= \min_{\eta: \eta(-1)=\omega_{\pi}(-1)} \min \{ m\geq 0: \pi_\eta^{K_0(m)} \neq 0 \text{ or } \pi_\eta^{K'_0(m)} \neq 0 \},
\]
where $\eta$ runs over all characters of $\Z_l^*$ with $\eta(-1)=\omega_{\pi}(-1)$.

\begin{prop}[{\cite[Theorem 3.4.1]{LR}}] Let $\pi$ be an irreducible admissible representation of $\SL_2(\Q_l)$. 
The conductor $c(\pi)$ depends only on the $L$-packet containing $\pi$. 
It agrees with the conductor $c(\til{\pi})$ of a minimal representation $\til{\pi}$ of $\GL_2(\Q_l)$ such that $\pi$ occurs in the restriction $\til{\pi}|_{\SL_2(\Q_l)}$. \end{prop}
Here a representation $\til{\pi}$ of $\GL_2(\Q_l)$ is called minimal if $c(\til{\pi}\otimes \chi) \geq c(\til{\pi})$ for all characters $\chi$ of $\Q_l^*$. 
The proposition lets us talk about the conductor of a local $L$-packet $\Pi$ and write $c(\Pi)$ for it.

In \cite{LR}, the authors study the various spaces $\pi_\eta^{K_0(m)}$ and $\pi_\eta^{K'_0(m)}$ for $\pi$ varying in a given $L$-packet $\Pi$ and all $m \geq c(\Pi)$. They determine the dimensions of these spaces depending on the type of the $L$-packet. 
Contrary to the $\GL_2(\Q_l)$ case (see \cite{Casselman}), the first non-zero space $\pi_\eta^{K_0(c(\pi))}$ or $\pi_\eta^{K'_0(c(\pi))}$ is not necessarily always one-dimensional. 

We summarize the results of \cite{LR} for supercuspidal $L$-packets. For that let $\Pi$ be a $L$-packet arising from an irreducible supercuspidal representation $\til{\pi}$ of $\GL_2(\Q_l)$. One has to distinguish between two kinds of such $L$-packets. 
The first kind (called \textit{unramified supercuspidal $L$-packet} in \cite{LR}) is given by representations $\til{\pi}$ that are compactly induced from a representation $\til{\sigma}$ of $\til{Z}\GL_2(\Z_l)$. Here $\til{Z}$ denotes the center of $\GL_2(\Q_l)$. The second kind (called \textit{ramified supercuspidal $L$-packet} in \cite{LR}) is given by representations $\til{\pi}$ that are compactly induced from the normalizer $N_{\GL_2(\Q_l)}\til{I}$ of the standard Iwahori $\til{I}$ of $\GL_2(\Q_l)$ of upper triangular matrices mod $l$. Supercuspidal $L$-packets of the second kind are always of size two, those of the first kind can be of size four, namely when the representation $\til{\sigma}$ is of level one in the sense of Definition 3.3.1 in \cite{LR}. But when this level is greater than or equal to two, the resulting $L$-packet is of size two. We will avoid local packets of size four below. 

\begin{prop}\label{newforms}
Let $\Pi=\{\pi_1,\pi_2\}$ be a supercuspidal $L$-packet of cardinality two with conductor $c(\Pi)$. 
Let $n$ be $\left \lfloor{c(\Pi)/2}\right\rfloor$. Let $\eta$ be a character of $\Z_l^*$with $c(\eta) \leq n$. 
Then 
\begin{enumerate}
	\item For $\pi \in \Pi$, $ \dim (\pi_\eta^{K_0(c(\Pi))}) \in \{0,1,2\}$. If $\Pi$ is of the first kind, we have 
	\[\dim ((\pi_1)_\eta^{K_0(c(\Pi))})=0  \text{ and } \dim ((\pi_2)_\eta^{K_0(c(\Pi))})=2 
	\] 
	or vice versa. If $\Pi$ is of the second kind, then  
	\[\dim ((\pi_1)_\eta^{K_0(c(\Pi))})=\dim ((\pi_2)_\eta^{K_0(c(\Pi))})=1.
	\] 
	In particular we always have
	\[\sum_{\pi \in \Pi} \dim (\pi_\eta^{K_0(c(\Pi))}) = 2.
	\]
\item More generally for any $m \geq c(\Pi)$, 
\[\sum_{\pi \in \Pi} \dim (\pi_\eta^{K_0(m)}) =2(m-c(\Pi)+1).\] 
\end{enumerate}
\end{prop}
\begin{proof} This follows from \cite[Proposition 3.3.4]{LR} and \cite[Proposition 3.3.8]{LR}. 
\end{proof}

Note also that for a smooth irreducible representation $\pi$ of $\SL_2(\Q_l)$ 
\[\pi^{K_1(m)} = \bigoplus_{\eta} \pi_\eta^{K_0(m)},
\] 
where $\eta$ runs through all characters of $K_0(m)/K_1(m)$. This implies that if $\pi^{K_1(m)}\neq~0$ for some $m\geq 0$, then the conductor $c(\pi)$ of $\pi$ is less than or equal to $m$.

\section{Symmetric square lifting}
We recall some background on the local symmetric square lifting as constructed by Gelbart and Jacquet in \cite{lifts}. In this section, $F$ is a local non-archimedean field of residue characteristic not $2$, with residue field $\mathbb{F}$ of size $q$. 
 
\begin{lem}[{\cite[Prop.\ 3.3]{lifts}}] Let $\sigma$ be an irreducible admissible representation of $\GL_2(F)$. 
Then there exists an irreducible admissible representation $\pi$ of $\GL_3(F)$ unique up to isomorphism, called the \textit{lift of} $\sigma$, such that 
\begin{enumerate}
	\item the central character of $\pi$ is trivial,
	\item $\pi$ is self-dual,
	\item for any character $\chi$ of $F^*$, 
	\[ L(s,\pi\otimes \chi) = L(s,(\sigma\otimes \chi) \times \widetilde{\sigma})/L(s,\chi) \text{ and }\]
	\[\varepsilon(s, \pi\otimes \chi; \psi) = \varepsilon(s,(\sigma \otimes \chi) \times \widetilde{\sigma}; \psi)/\varepsilon(s,\chi;\psi).
	\]
\end{enumerate}
\end{lem}
For the definition of $L$-and $\varepsilon$-factors see \cite{GJ}. We give the recipe for the lift of $\sigma$. In the following, $B \subset \GL_3(F)$ denotes the Borel subgroup of upper triangular matrices, and $P \subset \GL_3(F)$ the parabolic subgroup of block form $(2,1)$. 

\begin{description}
	\item[Case (I)]  $\sigma= \Ind(\mu_1, \mu_2)$ for two characters $\mu_i:F^*\rightarrow \C^*$, $i=1,2$. Write $\mu_i=\chi_i |.|^{t_i}_F$  where $|x|_F=q^{-v(x)}$, $t_i \in \R$, and $\chi_i:F^*\rightarrow \C^*$ satisfies $\chi \overline{\chi}=1$.
If $t_1-t_2=0$, the representation 
\[\pi:=\Ind^{\GL_3(F)}_B(\mu_1\mu_2^{-1},1,\mu_2\mu_1^{-1})
\] 
is irreducible and a lift of $\sigma$.
If $t_1-t_2 \neq 0$, we may assume $t_1>t_2$. Then the representation $\Ind^{\GL_3(F)}_B(\mu_1\mu_2^{-1},1,\mu_2\mu_1^{-1})$ admits a maximal subrepresentation $\tau$ and the lift of $\sigma$ is given by the quotient 
\[\pi:=\Ind^{\GL_3(F)}_B(\mu_1\mu_2^{-1},1,\mu_2\mu_1^{-1})/\tau.
\]
\item[Case(II)] Assume $\sigma$ is a supercuspidal representation, so $\sigma = \pi(\chi)$ for some character $\chi:E^*\rightarrow \C^*$ of a quadratic extension $E/F$. As before let $\overline{\chi}$ denote the conjugate of $\chi$ under the non-trivial element of $\tau \in \Gal(E/F)$, so $\overline{\chi}(x)=\chi(\tau{x})$. Define $\mu:= \chi \overline{\chi}^{-1}$ 
and let $\pi(\mu)$ be the representation of $\GL_2(F)$ attached to $\mu$ via the Weil-representation.
Then the lift of $\sigma$ is given by 
\[\pi=\Ind^{\GL_3(F)}_P(\pi(\mu), \varpi_{E/F}),\]
where $\varpi_{E/F}:F^*\rightarrow \C^*$ is the character associated to $E/F$ by local class field theory.

We recall that the representation $\pi(\mu)$ is supercuspidal if and only if $\mu$ does not factor through the norm $N^E_F$, i.e., if and only if $\chi \overline{\chi}^{-1}$ is not a quadratic character.
\item[Case(III)]
Finally, let $\sigma$ be a special representation, i.e., 
\[\sigma \cong \mathrm{St} \otimes \chi\] 
is a twist of the Steinberg representation of $\GL_2(F)$ by a character $\chi$ of $F^*$. Then the lift of $\sigma$ is given by the Steinberg representation $\pi$ of $\GL_3(F)$.
It is the square integrable component of 
\[\Ind^{\GL_3(F)}_B(|\cdot|,1,|\cdot|^{-1}).\] 
\end{description}

By construction the lifts of $\sigma$ and any twist $\sigma \otimes \chi$ agree. Therefore we get an induced map
\begin{eqnarray*}
\left\{L\text{-packets of } \SL_2(F)\right\} &\rightarrow& \left\{ \text{irreducible smooth representations of } \GL_3(F)\right\}\\
\Pi(\sigma) &\mapsto& \pi,
\end{eqnarray*}
that sends the $L$-packet $\Pi(\sigma)$ defined by a representation $\sigma$ of $\GL_2(F)$ to the lift $\pi$ of $\sigma$.

Recall from Remark \ref{sc2} that the $L$-packet $\Pi(\sigma)$ defined by a supercuspidal representation $\sigma \cong \pi(\chi)$ is of size two unless $\chi \overline{\chi}^{-1}$ is a quadratic character, in which case it is of size four. 
\begin{prop}\label{sym2}
Let $\sigma$ and $\sigma'$ be irreducible admissible representations of $\GL_2(F)$ and assume $\sigma = \pi(\chi)$ is supercuspidal and comes from a character $\chi$ such that $\chi \overline{\chi}^{-1}$ is not quadratic. Assume the lifts $\pi$ of $\sigma$ and $\pi'$ of $\sigma'$ are in the same Bernstein component. Then $\sigma'$ is supercuspidal and defines an $L$-packet of $\SL_2(F)$ which is of size two. 
\end{prop}
\begin{proof} By the previous discussion we see that $\pi$ is given by $\Ind^{\GL_3(F)}_P(\pi(\mu), \varpi_{E/F})$ for a supercuspidal representation $\pi(\mu)$. As the parabolic $P$ of block form $(2,1)$ and the Borel subgroup $B$ are not conjugate, the representation $\pi'$ has to be of the form 
\[\pi' \cong \Ind^{\GL_3(F)}_P(\pi(\mu'), \varpi_{E'/F})\] 
for a supercuspidal representation $\pi(\mu')$ which satisfies $\pi(\mu)\cong \pi(\mu')\otimes \eta$ for some unramified character $\eta$ of $F^*$. By the above description of the lifts we see that this is only possible if $\sigma'=\pi(\chi')$ is also supercuspidal and $\chi'$ is such that $\chi' \overline{\chi}'^{-1}$ is not quadratic.
\end{proof}

\section{Eigenvarieties a la Emerton}
Let $E/\Q_p$ be a finite extension with ring of integers $\mathcal{O}$ and residue field $\mathbb{F}_q$. Fix an embedding $E \hookrightarrow \Qbar_p$ as well as an isomorphism $\iota:\Qbar_p \stackrel{\sim}{\rightarrow} \C$. The main reference for this section is \cite{emerton}.

Fix a tame level, i.e., a compact open subgroup 
\[K^p=\prod_{l}K_l \subset \SL_2(\A_f^p)
\] 
and let $S(K^p)$ be the minimal set of primes $l$, such that $K_l=\SL_2(\Z_l)$ for all $l \notin S(K^p)$, $l \neq p$.

We write $\mathcal{H}(K^p):=C^\infty_c(\SL_2(\A_f^p)\dslash K^p)$ for the prime to $p$ Hecke algebra over $E$ of level $K^p$, where the symbol $\dslash$ means left-and right invariance. It decomposes 
\[\mathcal{H}(K^p) \cong \mathcal{H}(K^p)^{\ram}\otimes_E \mathcal{H}(K^p)^{\ur},\] 
where 
\[\mathcal{H}(K^p)^{\ur} = \bigotimes{'}_{l \notin S(K^p)\cup\{p\}}  C^{\infty}_c(\SL_2(\Q_l)\dslash\SL_2(\Z_l))
\] 
and $\mathcal{H}(K^p)^{\ram}=C^\infty_c(\SL_2(\A_{S(K^p)})\dslash K_S)$. 

For $R$ equal to either $\cO, \cO/\varpi^s \cO$ for some $s>0$ or to $E$ and $i\geq 0$, define 
\[H^i(K^p,R):=\varinjlim_{K_p} H^i(Y(K_pK^p),R),
\]
where the direct limit runs over all compact open subgroups $K_p \subset \SL_2(\Q_p)$ and as above
\[Y(K_pK^p)= \SL_2(\Q)\backslash \SL_2(\A)/K_pK^p \mathrm{SO}_2(\R)
\]
is the symmetric space of level $K_pK^p$. Recall that the completed cohomology of tame level $K^p$ is defined as
\[\til{H}^i(K^p,\cO):=\varprojlim_s H^i(K^p,\cO/\varpi^s \cO) = \varprojlim_s \varinjlim_{K_p} H^i(Y(K_pK^p),\cO/\varpi^s \cO).
\]
Then
\[\til{H}^i(K^p):= \til{H}^i(K^p,\cO) \otimes_\cO E
\]
is an $E$-Banach space equipped with an action of $\cH({K^p})$ as well as an admissible continuous action of the locally analytic group $\SL_2(\Q_p)$ in the sense of \cite{emerton2}.
Taking the limit
\[\til{H}^i:= \varinjlim \til{H}^i(K^p),
\]
we get a locally convex topological $E$-vector space together with an admissible continuous action of $\SL_2(\A_f)$.
We recover the previous space by taking $K^p$-fixed vectors:
\[(\til{H}^i)^{K^p}= \til{H}^i(K^p).\]

The only interesting cohomology index for the group $\SL_2/\Q$ is $i=1$. As for $K_p$ small enough the $Y(K_pK^p)$ are connected open Riemann surfaces, the cohomology groups $H^i(K^p,R)$ vanish for all $i \geq 2$. Furthermore $H^0(K^p,R) \cong R$.
For any inclusion $K_p \subset K_p'$ of compact open subgroups of $\SL_2(\Q_p)$, the transition maps 
\[H^0(Y(K'_pK^p),R)\cong R \rightarrow H^0(Y(K_pK^p),R) \cong R\]
are the identity, therefore 
\[\til{H}^0(K^p,\cO) \cong \cO.\]
Furthermore both $\til{H}^0(K^p,\cO)$ and $\til{H}^1(K^p,\cO)$ agree with the $p$-adic completions
\[\widehat{H}^i(K^p,\cO) := \varprojlim_s H^i(K^p,\cO)/\varpi^s,
\]
and are $\cO$-torsion free. 

The abstract ``unramified'' Hecke algebra over $\cO$, i.e., the polynomial algebra 
\[\mathbb{H}:=\cO[t_l:l\notin S(K^p)\cup \{p\}]
\]
acts on the spaces $H^1(Y(K_pK^p),\cO)$.  

Let $\til{\mathbb{H}}(K_p)$ be the image of $\mathbb{H}$ in the endomorphism ring of $H^1(Y(K_pK^p),\cO)$, equipped with its $p$-adic topology.
The inverse limit 
\[\til{\mathbb{H}}:= \varprojlim_{K_p} \til{\mathbb{H}}(K_p),
\]
equipped with the projective limit topology is a semi-local and noetherian ring, in particular, it has only finitely many maximal ideals. It acts faithfully on $\til{H}^1(K^p,\cO)$ and we have decompositions
\begin{equation}\label{dec1}
\til{H}^1(K^p,\cO) = \bigoplus_{\mathfrak{m}\in \mathrm{MaxSpec}(\til{\mathbb{H}})}\til{H}^1(K^p,\cO)_{\mathfrak{m}}
\end{equation}
and 
\[\til{H}^1(K^p) = \bigoplus_{\mathfrak{m}}\til{H}^1(K^p)_{\mathfrak{m}}\]
where in the last line  
\[\til{H}^1(K^p)_{\mathfrak{m}} := \til{H}^1(K^p,\cO)_{\mathfrak{m}}\otimes_\cO E.
\]

The eigenvarieties we are working with are defined relative to the so called weight space: Let $T\subset \SL_2(\Q_p)$ be the diagonal torus and let $\widehat{T}$ be the weight space as in \cite[Section 6.4]{emerton2}.
This is the rigid analytic space over $E$ that parametrizes the locally $\Q_p$-analytic characters of the group $T$, i.e., for any affinoid $E$-algebra $A$ we have 
\[ \widehat{T}(A)=\Hom_{la}(T(\Q_p),A^*).
\]

Emerton's eigenvarieties $p$-adically interpolate systems of Hecke eigenvalues coming from cohomological automorphic representations. 
 
Let $W=\Sym^k(E^2)$ be the algebraic representation over $E$ of $\SL_2$ of highest weight $k$. Note that $W$ is self-dual. 
Define 
\[H^1(K^p,W)=\varinjlim_{K_p} H^1(Y(K_pK^p),\mathcal{W}),
\]
where $\mathcal{W}$ is the local system on $Y(K_pK^p)$ attached to $W$. 
Let 
\[H^1(W)=\varinjlim_{K_f} H^1(Y(K_f),\mathcal{W}),
\]
where the direct limit runs over all compact open subgroups $K_f \subset \SL_2(\A_f)$.  

Let $\pi_f$ be an irreducible $\SL_2(\A_f)$-representation appearing as a subquotient of 
$H^1(W)\otimes_E \Qbar_p$ and such that $\pi_p$ embeds into the parabolic induction $\Ind_B^{\SL_2(\Q_p)}(\chi)$ from the Borel subgroup $B$ of upper triangular matrices and for some smooth $\Qbar_p$-valued character $\chi$ of $T$. Assume $\pi_f^{K^p}\neq 0$. Since $\pi_f$ is irreducible the spherical Hecke algebra $\mathcal{H}(K^p)^{\ur}$ acts on $\pi_f^p$ via a $\Qbar_p$-valued character $\lambda$. Let $\psi_W$ denote the highest weight character of $W$ regarded as a character of $T$. The pair $(\chi \psi_W, \lambda)$ then defines a point of the locally ringed space $\hat{T}\times \Spec \mathcal{H}(K^p)^{\ur}$. Following Emerton we call such a point a classical point. This is justified as the representation $\pi_f$ giving rise to such a point is the finite part of an automorphic representation $\pi$ of $\SL_2$. We remark at this point that as we only work with algebraic automorphic representation the finite part $\pi_f$ of such a representation can always be defined over $\Qbar$ and in fact over a number field. 

In \cite{emerton}, Emerton constructs an eigenvariety of tame level $K^p$ for the $\cH(K^p)\times \SL_2(\Q_p)$ module $\til{H}^1(K^p)$. We can apply his construction to any of spaces 
\[\til{H}^1(K^p)_{\mathfrak{m}}\]
occurring in the decomposition (\ref{dec1}) above. The outcome of his construction is a possibly non-reduced eigenvariety $E^1(K^p,\mathfrak{m})$ equipped with a coherent sheaf $\cM^1$ of $\mathcal{H}(K^p)^{\ram}$-modules. We pass to the reduced space $E(K^p,\mathfrak{m}):=E^1(K^p,\mathfrak{m})^{red}$ and let $\cM$ be the pullback of $\cM^1$ to $E(K^p,\mathfrak{m})$. Emerton's construction furthermore provides us with an injective map of locally ringed spaces 
\[E(K^p,\mathfrak{m}) \rightarrow \hat{T}\times \Spec \mathcal{H}(K^p)^{\ur}.
\]    

\begin{thm}[\cite{emerton}, \cite{hill}]\label{evs} 
The eigenvariety $E(K^p,\mathfrak{m})$ is a reduced rigid analytic space which has the following properties.
\begin{enumerate}
	\item Projection onto the first factor induces a finite map $E(K^p,\mathfrak{m}) \rightarrow \hat{T}$.
	\item If $(\phi, \lambda)$ is a point on $E(K^p,\mathfrak{m})$ such that $\phi$ is locally algebraic and of non-critical slope (in the sense of \cite[Definition 4.4.3]{emerton-jacquet1}), then $(\phi,\lambda)$ is a classical point.
	\item The coherent sheaf $\mathcal{M}$ of $\mathcal{H}(K^p)^{\ram}$-modules over $E(K^p,\mathfrak{m})$ satisfies the following property. For any classical point $(\chi \psi_W,\lambda) \in E(K^p,\mathfrak{m})$ of non-critical slope, the fibre of $\mathcal{M}$ over the point $(\chi \psi_W, \lambda)$ is isomorphic as a $\mathcal{H}(K^p)^{\ram}$-module to the dual of the $(\chi,\lambda)$-eigenspace of the Jacquet module of the smooth representation $H^1(K^p,W)$.
\end{enumerate}
\end{thm}
\begin{proof} This follows from Theorem 0.7 of \cite{emerton}. The $\SL_2(\A_f)$-equivariant edge map 
\[H^1(W)\otimes_E W \rightarrow \til{H}^1_{W-\mathrm{lalg}}\]
from \cite{emerton} Corollary 2.2.18 and Remark 2.2.19 is an isomorphism by \cite{hill}. 
\end{proof}

For a classical point $z = (\chi\psi_W, \lambda) \in E(K^p,\mathfrak{m})(\Qbar_p)$ coming from a cuspidal automorphic representation $\pi(z)$ define the classical subspace  
\[\cM^{cl}_{\zbar} := \Hom ( J_B(H^1(K^p,W))^{(\chi, \lambda)}, k(z)) \subset \cM_{\zbar}.
\] 
By the results recalled in Section \ref{sec: background} we have 
\begin{eqnarray*}
&& J_B(H^1(K^p,W))^{(\chi, \lambda)}\otimes_{k(z),\iota}\C\\
 &\cong & J_B(\varinjlim_{K_p} H^1(K_pK^p,W)^\lambda)^{\chi} \otimes_{k(z),\iota}\C \\
&\cong & J_B(\varinjlim_{K_p} H_{par}^1(K_pK^p,W)^\lambda)^\chi \otimes_{k(z),\iota}\C \\
& \cong & \bigoplus_{\substack{\pi \ adm.\ rep. :\\ \pi_l \cong \pi(z)_l \\ \forall l \notin S(K^p) \cup \{p\}}} m(\pi) (\pi_{S(K^p)})^{K^\ram} \otimes J_B(\pi_p)^{\chi} \otimes H^1_{rel.Lie}(\mathfrak{g}, K_\infty, W_\C \otimes \pi_\infty).
\end{eqnarray*}
In the last line we have abused notation, e.g., as we have written $\chi$ instead of $\iota \circ \chi$. We will continue to suppress it from the notation when we change coefficients from $\Qbar_p$ to $\C$ using the isomorphism $\iota$ as it will always be clear from the context. 

In the next section we work with certain direct summands of the sheaf $\cM$. For that first note the following easy lemma.
\begin{lem}\label{cohsum} Let $\cM$ be a coherent $\mathcal{H}(K^p)^{ram}$-module on a rigid space $X$. Let $e \in \mathcal{H}(K^p)^{ram}$ be an idempotent. 
Then $e\cM$ is coherent.  
\end{lem}
\begin{proof} As $e$ is an idempotent we have $\cM \cong e\cM \oplus (1-e)\cM$, and direct summands of coherent sheaves are coherent.
\end{proof}

The idempotents we use arise as follows. The subgroups $K_{1}(m_l)$ and $K_0(m_l)$ of $\SL_2(\Z_l)$ were defined in Section \ref{sec: newforms} for any integer $m_l \geq 0$. Recall that any character 
\[\eta_l: \Z_l^*/(1+l^{m_l}\Z_l)\rightarrow \Qbar_p^*
\]
defines a character of $K_0(m_l)$, by
\[ \eta_l \left(\left(\begin{array}{cc}	a & b \\ c & d \end{array} \right)\right) = \eta_l(d).
\]

Now let $K^p \subset \SL_2(\A_f^p)$ be a tame level such that $K_l=K_1(m_l)$ for all $l \in S(K^p)$. Choose characters $\eta_l$ as above and let 
\[\eta:= \prod_{l \in S(K^p)}{\eta_l}: \prod_l K_0(m_l) \rightarrow \Qbar_p^*
\] 
be their product. 
Then if $E \subset \Qbar_p$ contains the values of $\eta$, there exists an idempotent $e_{\eta}$ in $\mathcal{H}(K^p)^{\ram}$ associated with the representation $\eta$ of $K_0^\ram:=\prod_{l \in S(K^p)} K_0(m_l)$. 
Concretely we choose the Haar measure on $\SL_2(\A_{S(K^p)})$ such that $\mu_l(K_0(m_l))= 1$ for all $l \in S(K^p)$. 
Then let $e_{\eta}:\SL_2(\A_{S(K^p)})\rightarrow E$ be defined as
\[ e_{\eta}(g)= 
\begin{cases} 
    \eta^{-1}(g),& \text{if } g \in K_0^\ram \\
    0,              & \text{otherwise.}
\end{cases}
\]  
One easily checks that this defines an idempotent and that for two distinct characters $\eta \neq \eta'$ the resulting idempotents are orthogonal. For any smooth representation $\pi$ of $\SL_2(\A_{S(K^p)})$ with coefficients in $\Qbar_p$, $e_{\eta}$ is the projection onto the space $\pi_\eta^{K_0^\ram}$.

In the next section we need some extra properties that the eigenvarieties $E(K^p,\mathfrak{m})$ satisfy, at least when the maximal ideal $\mathfrak{m}$ is suitably chosen. In the rest of this section we will explain and verify these properties.

Recall that a rigid space $X/\Sp(E)$ is called nested if it has an admissible cover by open affinoid subspaces $\{X_i\}_{i\geq 0}$ such that $X_i \subset X_{i+1}$ for all $i \geq 0$ and the natural $E$-linear map $\cO(X_{i+1}) \rightarrow \cO(X_i)$ is compact. 
\begin{lem} The eigenvariety $E(K^p, \mathfrak{m})$ is nested. In particular, 
\[\cO(E(K^p,\mathfrak{m}))^0:=\{f \in \cO(E(K^p,\mathfrak{m})) : |f(x)|\leq 1 \ \forall x \in E(K^p,\mathfrak{m})\}\] 
is compact.
\end{lem}
\begin{proof} The space $\hat{T} \cong \Hom_{la}(T(\Z_p), \mathbb{G}_m) \times \mathbb{G}_m$ is a product of nested spaces, therefore nested. The map $E(K^p) \rightarrow \hat{T}$ is finite. By \cite[Lemma 7.2.11]{BC} this implies that $E(K^p)$ is also nested. As $E(K^p)$ is reduced this implies the compactness assertion, again by \cite[Lemma 7.2.11]{BC}.  
\end{proof}

\begin{defn}
Define the ideal $\mathfrak{m}_{0}$ to be the kernel of the map
\[\mathbb{H} \rightarrow \mathbb{F}_q\]
that sends $t_l$ to $l^2+l$. 
\end{defn}
\begin{rem}
This ideal comes from the fact that $t_l \in \mathbb{H}$ acts on $H^0(K^p,\cO) \cong \cO$ by $x \mapsto (l^2+l)x$, in particular 
\[ H^0(K^p,\cO)_\mathfrak{m} = 0,
\]
for all $\mathfrak{m} \in \mathrm{MaxSpec}(\mathbb{H}), \mathfrak{m} \neq \mathfrak{m}_0$. 
\end{rem}

For an open subgroup $H\subset \SL_2(\Q_p)$ let $\mathcal{C}^{la}(H,E)$ denote the locally analytic $E$-valued functions on $H$. 
\begin{lem}\label{inj} Let $\mathfrak{m} \in \mathrm{MaxSpec}(\mathbb{H}), \mathfrak{m}\neq \mathfrak{m}_{0}$ be a maximal ideal with $\til{H}^1(K^p,\cO)_{\mathfrak{m}} \neq 0$. There is a compact open subgroup $H \subset \SL_2(\Q_p)$ such that 
\[ (\til{H}^1(K^p,\cO)_{\mathfrak{m}})_{la}\stackrel{\sim}{\rightarrow} \mathcal{C}^{la}(H,E)^r,
\]
for some $r \geq 1$ as representations of $H$. 
\end{lem}
\begin{proof} This is proved as in \cite{emertonlg}. We have $\widehat{H}^1(K^p,\cO) \cong \til{H}^1(K^p,\cO)$, and 
\[\widehat{H}^1(K^p,\cO)_{\mathfrak{m}}/\varpi^s \widehat{H}^1(K^p,\cO)_{\mathfrak{m}} \cong H^1(K^p, \cO/\varpi^s \cO)_{\mathfrak{m}}.
\]
By the proof of Corollary 5.3.19 in \cite{emertonlg} it suffices to show that we can find an open subgroup $H$ such that $H^1(K^p, \cO/\varpi^s \cO)_{\mathfrak{m}}$ is injective for all $s\geq 1$. This last claim can be proved in the same way as Proposition 5.3.15 in \cite{emertonlg}. Our assumption $\mathfrak{m} \neq \mathfrak{m}_0$ implies that all the relevant $H^0$-terms vanish.
\end{proof}

\begin{lem}\label{zda} Let $\mathfrak{m} \neq \mathfrak{m}_0$ be a maximal ideal in $\mathbb{H}$ as above. The eigenvariety $E(K^p,\mathfrak{m})$ is equidimensional of dimension one. It contains a Zariski-dense set of classical points, which accumulates at any of its points. 
\end{lem}
\begin{proof} Define $\widehat{T}_0:= \Hom_{la}(T(\Z_p),\mathbb{G}_m)$, \footnote{This is the weight space considered in the work of Coleman and Buzzard.} then $\widehat{T}\cong \widehat{T}_0\times_E \mathbb{G}_m$. Write $\widehat{T}_0$ as an increasing union of affinoid opens $\Sp(A_n)$, $n\geq 0$. 
 
Lemma \ref{inj} allows us to apply the results from the proof of Proposition 4.2.36 in \cite{emerton-jacquet1}. 
In particular, we get that locally over $\Sp(A_n) \subset \widehat{T}_0$, there is a finite map  
\[E(K^p,\mathfrak{m})_n \rightarrow Z_n\]      
from the eigenvariety $E(K^p,\mathfrak{m})_n:= E(K^p,\mathfrak{m}) \times_{\widehat{T}_0} \Sp(A_n)$ to the spectral variety $Z_n \hookrightarrow \mathbb{G}_m \times_E \Sp(A_n)$ attached to the operator $u_p:=\left(^p \ _{p^{-1}}\right)\in T(\Q_p)$ which acts compactly on the respective Banach space. We refer to the proof of Corollary~4.1 in \cite{newton2}, where a similar argument is explained nicely in the context of $\GL_2$. In fact reinterpreting all the $\GL_2$-specific notation occurring there in terms of $\SL_2$ and replacing the representation $V$ by  $\til{H}^1(K^p)_\mathfrak{m}$, everything goes through. 
In particular, for every point on $E(K^p,\mathfrak{m})$ we may find an arbitrary small open affinoid neighbourhood $U$ and an open subspace $V \subset \widehat{T}_0$ such that the map 
\[U \rightarrow V\] 
obtained as the restriction of $E(K^p,\mathfrak{m})\rightarrow \widehat{T}_0$ to $U$, is finite and surjective when restricted to each irreducible component.
We may therefore deduce the Zariski-density and accumulation property as in \cite[Section 6]{chenevier} from the fact that the classical weights $\{k\geq 2\}\subset \widehat{T}_0$ are Zariski-dense and accumulation in weight space $\widehat{T}_0$. 
\end{proof}

From Theorem \ref{evs} above we get a morphism $\psi: \mathcal{H}(K^p)^{\ur} \rightarrow \cO(E(K^p,\mathfrak{m}))$. Abbreviate $S:=S(K^p)$. 
Let $t_l \in \mathcal{H}(K^p)^{\ur}$ be the characteristic function on the double coset
\[
\SL_2(\widehat{\Z}^S) \left(\begin{array}{cc} l & 0\\0 & l^{-1}\end{array}\right) \SL_2(\widehat{\Z}^S),
\]
where $\left(\begin{smallmatrix} l& \\ & l^{-1} \end{smallmatrix}\right)$ is understood to be the matrix in $\SL_2(\widehat{\Z}^S)=\prod_{q\notin S} \SL_2(\Z_q)$ which is equal to $1$ for all $q \neq l$ and equal to $\left(\begin{smallmatrix} l& \\ & l^{-1} \end{smallmatrix}\right)$ at $l$. 
 
\begin{lem} Let $E(K^p,\mathfrak{m})$ be the eigenvariety of tame level $K^p$. Let $G_{\Q,S} $ be the Galois group of a maximal extension of $\Q$ unramified outside $S$. 
There exists a three dimensional pseudorepresentation 
\[ T: G_{\Q,S} \rightarrow \mathcal{O}(E(K^p,\mathfrak{m}))
\]
such that $T(\operatorname{Frob}_l) = \psi\left(\frac{1}{l}(t_l+1)\right)$, for all $l \notin S$.
\end{lem}
\begin{proof} We follow the strategy of Proposition 7.1.1 in \cite{chenevier}. The fact that 
\[\psi\left(\frac{1}{l}(t_l+1)\right) \in \cO(E(K^p,\mathfrak{m}))^0
\] 
comes from the definition of the action of $\cH(K^p)^{ur}$ on $\til{H}^1(K^p)_\mathfrak{m}$ via correspondences. Then Lemma \ref{zda} implies that we may apply \cite[Prop.7.1.1]{chenevier} to the Zariski-dense subset of classical points and the three-dimensional pseudorepresentations that are described in the proof of Lemma 2.2 of \cite{p-adicpackets}.
\end{proof}

\begin{lem} Let $T$ be a pseudorepresentation on $E(K^p,\mathfrak{m})$. Then for $l \in S$, $T|_{I_l}$ is constant on connected components of $E(K^p,\mathfrak{m})$.
\end{lem}
\begin{proof} \cite{BC} Lemma 7.8.17.
\end{proof}

\begin{prop}\label{rigidity}
Let $x \in E(K^p,\mathfrak{m})(\Qbar_p)$ be a classical point and choose an automorphic representation $\pi_x$ which gives rise to $x$. Assume for some $l \neq p$, $\pi_{x,l}$ is supercuspidal and that the $L$-packet of $\pi_{x,l}$ contains exactly two elements. 
Let $y \in E(K^p,\mathfrak{m})(\Qbar_p)$ be another classical point which is on the same connected component as $x$. 
Then any automorphic representation $\pi_{y}$ giving rise to $y$ has the property that $\pi_{y,l}$ is supercuspidal, and the $L$-packet containing $\pi_{y,l}$ is again of size two.
\end{prop}

\begin{proof} Choose lifts $\til{\pi}_x$ and $\til{\pi}_y$ to $\GL_2$ that are unramified at all places not in $S$. Then $T_x$ is the trace of the representation $\tau_{x}:=\Sym^2(\rho(\til{\pi}_x))\otimes \det(\rho(\til{\pi}_x))^{-1}$, where $\rho(\til{\pi}_x)$ is the 2-dimensional Galois representation associated to $\til{\pi}_x$ by Deligne and similarly for $y$. 
As ${T_x}|_{I_l}= {T_y}|_{I_l}$, the lifts $\mathrm{rec}^{-1}(\tau_{x,l})$ of $\til{\pi}_{x,l}$ and $\mathrm{rec}^{-1}(\tau_{y,l})$ of $\til{\pi}_{y,l}$ to $\GL_3(\Q_l)$ are in the same Bernstein component and we can apply Proposition \ref{sym2} to get the result.  
\end{proof}

\section{Existence of non-classical $p$-adic automorphic forms}\label{sec: mainthm}

Let $p\geq 7$ be a prime.
Let $\Pi(\theta)$ be an endoscopic $L$-packet attached to an algebraic character 
\[\theta: \mathrm{Res}_{F/\Q}\mathbb{G}_m(\A) \rightarrow \C^*
\] 
of the adelic points of the elliptic endoscopic group $\mathrm{Res}_{F/\Q}\mathbb{G}_m$, where $F/\Q$ is an imaginary quadratic field in which $p$ splits. 
Let $S$ be the set of finite primes $l$, where $\Pi(\theta)_l$ is ramified. 
We say that $\Pi(\theta)$ satisfies \textit{Hypothesis $(\star)$} if
\begin{enumerate}
	\item  $\Pi(\theta)_p$ is a singleton consisting of a $\SL_2(\Z_p)$-unramified representation, 
	\item for all $l \in S$, $\Pi(\theta)_l$ is a supercuspidal $L$-packet of size two, 
	\item $2 \notin S$ and
  \item $\Pi(\theta)_{\infty}=\{D_{k+1}^\pm\}$ is a discrete series $L$-packet of weight $k+1\geq 1$.
\end{enumerate}
We refer the reader to \cite[Lemma 4.1]{p-adicpackets} for the construction of examples of such $L$-packets. 
\begin{lem}\label{mult} Let $\tau_f \in \Pi(\theta)_f$ be any element. 
Then exactly one of the representations 
$\tau_f \otimes D^+_{k+1}$ and $\tau_f \otimes D^{-}_{k+1}$ is automorphic. 
\end{lem}
\begin{proof} This follows from the multiplicity formulae in \cite{LL}; more precisely we are in the situation of Proposition 6.7 in loc.cit.
\end{proof}

Let $c_l:=c(\Pi(\theta)_l)$  be the conductor of the local $L$-packet $\Pi(\theta)_l$ as in Section \ref{sec: newforms}.
From now on we fix the tame level to be 
\[ K^p:= \prod_{l \in S } K_1(c_l) \times \prod_{l \notin S\cup \{p\}}\SL_2(\Z_l) \subset \SL_2(\A_f^p).
\]
We also fix a coefficient field $E/\Qbar_p$ big enough to contain the values of all characters $\eta$ occurring below. For the main theorem it suffices that $E$ contains the $(l-1)$th roots of unity for all $l \in S$. 

Furthermore we fix $\tau_f \in \Pi(\theta)_f$ such that $(\tau^p_f)^{K^p} \neq 0$. 
Let 
\begin{equation}\label{tau'}
\tau' \in \{\tau_f \otimes D^+_{k+1}, \tau_f \otimes D^-_{k+1}\}
\end{equation}
be the unique representation with $m(\tau')=1$. Then $\tau'$ gives rise to a maximal ideal $\mathfrak{m} \subset \mathbb{H}$ as follows: Let $\til{\tau}'$ be a lift of $\tau'$ to $\GL_2(\A)$ which is unramified at all $l \notin S$. Attached to $\til{\tau}'$ there is a two dimensional Galois representation $\rho: G_{\Q,S} \rightarrow \GL_2(\Qbar_p)$ constructed by Deligne. Let $\rhobar$ be the reduction of $\rho  \mod p$. Define $\mathfrak{m}$ to be the kernel of the map
\[\mathbb{H} \rightarrow \overline{\mathbb{F}}_p, \ t_l +1  \mapsto l \frac{\mathrm{tr}^2(\rhobar(\mathrm{Frob}_l))}{\det(\rhobar(\mathrm{Frob}_l))}. \]  

\begin{lem} The ideal $\mathfrak{m}$ is not equal to $\mathfrak{m}_{0}$.
\end{lem}
\begin{proof} The representation $\rho$ is induced from a character of $G_F$. Therefore the trace of $\rhobar$ vanishes for all primes $l$ that are inert in $F$, which is of density $1/2$. On the other hand the density of primes $l$ such that $l^2+l+1= 0 \mod p$ is $2/(p-1) < 1/2$, as $p\geq 7$. The claim follows.
\end{proof}

Let $z = (\chi \psi_W, \lambda) \in E(K^p, \mathfrak{m})(\Qbar_p)$ be any classical point coming from an automorphic representation $\pi$. We denote by $\Pi(z)$ the global $L$-packet containing $\pi$. Note that this is well-defined as $\lambda$ determines a unique representation $\pi_l$ of $\SL_2(\Q_l)$ at all places $l \notin S$ and for $\SL_2$ this is enough to determine the global $L$-packet. 

\begin{defn} Let $z$ be a classical point on $E(K^p, \mathfrak{m})$. We say $z$ is \emph{stable} if the $L$-packet $\Pi(z)$ defined by $z$ is a stable $L$-packet. Otherwise we call the point an \emph{endoscopic} point. 
\end{defn}
 
The representation $\tau'$ from (\ref{tau'}) above gives rise to two distinct points $x,y \in E(K^p, \mathfrak{m})(\Qbar_p)$, one of which, say $x$, is of critical slope whereas the point $y$ is of non-critical slope in the sense of \cite[Def.4.4.3]{emerton-jacquet1}. One verifies this easily using \cite[Lemma 4.4.1]{emerton-jacquet1}, a calculation like in Lemma 3.3 of \cite{p-adicpackets} and the fact that $\tau_p \cong \Ind_B^{\SL_2(\Q_p)}(\theta_w\theta^{-1}_{\overline{w}})$, where $w$ and $\overline{w}$ denote the two places in $F$ above $p$.

We say that a global $L$-packet $\Pi$ has property $\mathcal{P}(S)$ if 
\begin{itemize}
	\item the local $L$-packet $\Pi_l$ is unramified for all $l \notin S$, 
\item the local $L$-packet $\Pi_l$ is supercuspidal of size two for all $l \in S$. 
\end{itemize}
By Lemma \ref{rigidity} any classical point $z \in E(K^p,\mathfrak{m})$ on the same connected component as $x$ has property $\mathcal{P}(S)$. 

For a $L$-packet $\Pi$ we let $\omega_{\Pi_S}: \prod_{l \in S} Z(\SL_2(\Q_l))\rightarrow \{\pm 1\}$ be the product of the central characters of the representations $\pi_l \in \Pi_l$ for $l \in S$. 
The following proposition is the key calculation for our main theorem. 
Define $c:= \prod_{l \in S} c_l$, , where $c_l:=c(\Pi(\theta)_l)$ as before and let $K_0(c):=\prod_{l \in S} K_0(c_l)$.

\begin{prop}\label{fibres}
\begin{enumerate}
	\item Let $z \in E(K^p,\mathfrak{m})(\Qbar_p)$ be a stable point on the same connected component as the point $x$. Let $\Pi(z)$ be its associated $L$-packet. Let $\eta = \prod \eta_l:\prod_{l\in S} \Z_l^* \rightarrow \Qbar_p^*$ be a character such that $c(\eta_l) \leq c(\Pi(z)_l)$ for all $l \in S$. Then if $ \eta(-1) = \omega_{\Pi(z)_S}(-1)$ we have
\[\dim_{k(z)} e_{\eta}\cM^{cl}_{\zbar} \geq 2(\prod_{l\in S}2(c_l-c(\Pi(z)_l)+1)). \]
Otherwise $e_{\eta}\cM^{cl}_{\zbar}=0$. 
\item Let $\eta = \prod \eta_l:\prod \Z_l^* \rightarrow \Qbar_p^*$ be a character such that $c(\eta_l) \leq c_l$ for all $\l \in S$. Then 
\[\dim_{k(x)} e_{\eta}\cM^{cl}_{\xbar} = 2^{\# S} \]
if $ \eta(-1) = \omega_{\Pi(\theta)_S}(-1)$. Otherwise $e_{\eta}\cM^{cl}_{\xbar}=0$. 
\end{enumerate}
\end{prop}

\begin{proof} 
Given any classical point $z \in E(K^p,\mathfrak{m})(\Qbar_p)$ there is by definition an automorphic representation $\pi=\otimes \pi_l \in \Pi(z)$ with $(\pi_f^p)^{K^p} \neq 0$, and so for all $l \in S$, there exists $\eta'_l$ such that $(\pi_l)_{\eta'_l}^{K_0(c_l)}\neq 0$. Therefore $c(\Pi(z)_l)\leq c_l$.
First note that for $W=\Sym^k(E^2)$, we have
\begin{eqnarray*}
&& J_B(H_{\mathrm{par}}^1(K^p,W))\otimes_{E,\iota} \C \\
&=& J_B\left(\varinjlim_{K_p} \bigoplus_{\substack{\pi \ adm.\ rep.\\ of \SL_2(\A)}} m(\pi) \pi_f^{K^pK_p}\otimes H^1_{rel.Lie}(\mathfrak{g}, K_\infty, W_{\C}\otimes \pi_\infty)\right)\\
&=& \bigoplus_{\pi \ adm.\ rep} m(\pi) (\pi^p_f)^{K^p} \otimes J_B(\pi_p) \otimes H^1_{rel.Lie}(\mathfrak{g}, K_\infty, W_{\C}\otimes \pi_\infty).
\end{eqnarray*}
By Lemma \ref{(g,K)} above, the term 
\[
H^1_{rel.Lie}(\mathfrak{g}, K_\infty, W_{\C}\otimes \pi_\infty)
\]
 is non-zero if and only if $\pi_\infty = D_{k+1}^{\pm}$ and in this case it is isomorphic to $\C$.
Therefore for any classical point $z=(\chi \psi_W, \lambda) \in E(K^p,\mathfrak{m})(\Qbar_p)$ we have
\[
J_B(H^1(K^p,W))^{(\chi, \lambda)}\otimes_{k(z),\iota} \C= \bigoplus_{\substack{\pi \ adm.\ rep.\\ \pi_\infty = D_{k+1}^\pm \\ \pi^{S \cup\{p\}} \cong \pi_{\lambda}}} m(\pi) (\pi_{S})^{K^\ram} \otimes J_B(\pi_p)^{\iota(\chi)},
\]
where $k(z)$ is the residue field at $z$ and $\pi_\lambda$ is the representation of $\SL_2(\A_f^{S\cup\{p\}}) $ determined by $\iota \circ \lambda$. 
Here and below we write $\iota(\chi)$ for the composition $\iota \circ \chi$. 
 
We see that the direct sum runs over a subset of the global $L$-packet $\Pi(z)$ determined by $z$, namely over 
\[X(z):= \{\pi \in \Pi(z) | \pi_l=\pi_l^{0} \text{ for all } l \notin S \cup \{p\} \}.
\]
Here $\pi_l^0$ denotes the unique member of $\Pi(z)_l$, which has a fixed vector under $\SL_2(\Z_l)$. 
The following observation is crucial: If $z$ is a stable point then any $\pi \in X(z)$ is an automorphic representation, whereas for the point $x$ the number of automorphic representations in $X(x)$ is $\#X(x)/2$. 

Now let $\eta$ be as in the statement of the proposition and let $z=(\chi \psi_W,\lambda) \in E(K^p,\mathfrak{m})$ be a stable point on the connected component of $x$. Then
\begin{eqnarray*}
\dim_{k(z)} e_{\eta}\cM^{cl}_{\zbar}
 &=&  \dim_{k(z)} e_{\eta}\Hom(J_B(H^1(K^p,W))^{(\chi, \lambda)},k(z)) \\
&=& \sum_{\pi \in X(z)} {m(\pi) \dim_\C \left(\iota(e_{\eta})\left((\pi_{\ram})^{K^{\ram}} \otimes J_B(\pi_p)^{\iota(\chi)}\right)\right)}\\
&=& \sum_{\pi \in X(z)} {\dim_\C \left(\left(\iota(e_{\eta}) (\pi_{ram})^{K^{ram}}\right) \otimes J_B(\pi_p)^{\iota(\chi)}\right)}\\
&=& \sum_{\pi \in X(z)} \left(\prod_{l \in S} \dim_\C (\pi_l)^{K_0(c_l)}_{\iota(\eta_l)} \right) \dim_\C J_B(\pi_p)^{\iota(\chi)}, \\
\end{eqnarray*}
where in the second to last line we have used that any element $\pi \in X(z)$ is automorphic.
Define 
\[ X(z)^{+} = \{\pi \in X(z): \pi_\infty = D_{k+1}^+\}
\]
and $X(z)^{-} = \{\pi \in X(z): \pi_\infty = D_{k+1}^-\}$. 
Then
\begin{eqnarray*}
& & \dim_{k(z)} e_{\eta}\cM^{cl}_{\zbar} \\
&=& \sum_{\pi \in X(z)^{+}} \left(\prod_{l \in S} \dim_\C (\pi_l)^{K_0(c_l)}_{\iota(\eta_l)} \right) \dim_\C J_B(\pi_p)^{\iota(\chi)} \\ 
&& + \sum_{\pi \in X(z)^{-}} \left( \prod_{l \in S} \dim_\C (\pi_l)^{K_0(c_l)}_{\iota(\eta_l)}\right)\dim_\C J_B(\pi_p)^{\iota(\chi)}\\
&\geq & 2(\prod_{l\in S}2(c_l-c(\Pi(z)_l)+1)) \geq 2^{\# S +1}.
\end{eqnarray*}
where in the last line we have used Proposition \ref{newforms}. Furthermore the second to last inequality comes from the fact that $\Pi(z)_p$ might be of size two.

We now calculate the corresponding space at the point $x=(\chi(x) \psi_{W(x)}, \lambda(x)) \in E(K^p, \mathfrak{m})(\Qbar_p)$. Let $X(x)_{aut}\subset X(x)$ be the subset of automorphic representations. The key point in the following is that $X(x)_{aut} \neq X(x)$. We have
\begin{eqnarray*}
\dim_{k(x)} e_{\eta}\cM^{cl}_{\xbar} &=&  \dim_\C \left( e_{\eta} J_B(H^1(K^p,W(x)))^{(\chi(x), \lambda(x))}\otimes_{k(x)}\C \right) \\
&=& \sum_{\pi \in X(x)}{m(\pi) \dim_\C \left(\left(\iota(e_{\eta}) (\pi_{ram})^{K^{ram}}\right) \otimes J_B(\pi_p)^{\iota(\chi(x))}\right)}.\\
\end{eqnarray*}
By our assumptions the $L$-packet $\Pi(\theta)_p=\{\pi(\theta)_p\}$ is a singleton and by Lemma \ref{jacdim} 
\[\dim_\C J_B(\pi(\theta)_p)^{\iota(\chi(x))}=1.\]
Therefore
\begin{eqnarray*}\dim e_{\eta}\cM^{cl}_{\xbar}
&=& \sum_{\pi \in X(x)_{aut}} \prod_{l \in S} \dim (\pi_l)^{K_0(c_l)}_{\iota(\eta_l)}\\
&=& \prod_{l\in S} \sum_{\pi_l \in \Pi_l} \dim (\pi_l)^{K_0(c_l)}_{\iota(\eta_l)}\\
&=& \prod_{l\in S} 2 = 2^{\# S}.
\end{eqnarray*}
\end{proof}

\begin{thm}\label{mainthm} Let $\Pi(\theta)$ be an endoscopic $L$-packet of $\SL_2(\A)$ satisfying Hypothesis~$(\star)$. Let $\tau_f \in \Pi(\theta)_f$ be such that $\tau_f^{K^p}\neq 0$ and let $x$ be the point on $E(K^p, \mathfrak{m})$ of critical slope defined by the automorphic representation $\tau' \in \Pi(\theta)$ with $\tau'_f=\tau_f$. Then there exist non-zero non-classical forms in $\cM_{\xbar}$, i.e., 
\[\cM_{\xbar}/ \cM^{cl}_{\xbar} \neq 0.\]  
\end{thm}
\begin{proof}
Fix an affinoid open neighbourhood $U$ of $x$, such that $U$ contains a Zariski-dense set $Z$ of classical points. Such a neighbourhood exists by Lemma \ref{zda}, whose proof also provides us with the compact operator $u_p$. After possibly shrinking $U$ we may assume that the slope of $u_p$ is constant on $U$. By an argument as in the proof of \cite[Theorem 4.3]{p-adicpackets} we may assume that $\Pi(z)$ is stable for all $z \in Z, z \neq x$.  

Let $\mathcal{X}$ be the set of characters
\[ \prod_{l \in S} \Hom(\Z_l^*/(1+l\Z_l), \Qbar_p^*) \cong \prod_{l \in S} \mathbb{F}_l^*.
\]
As before we view $\eta \in \mathcal{X}$ as a character of the group $K_0(c)$. 
Define  
\[\cM':= \bigoplus_{\eta \in \mathcal{X}}e_{\eta}\cM . 
\]
By Lemma \ref{cohsum} this is a coherent sheaf on $E(K^p, \mathfrak{m})$ and a direct summand of $\cM$. 
For any classical point $z \in E(K^p,\mathfrak{m})(\Qbar_p)$, let 
\[(\cM'_{\zbar})^{cl}:= \cM_{\zbar}^{cl}\cap \cM_{\zbar}' = \bigoplus_{\eta \in \mathcal{X}} e_\eta \cM_{\zbar}^{cl}.
\] 
To prove the theorem it suffices to show that there exist non-classical forms in $\cM'_{\xbar}$.

Given any character
\[\mu: Z_S:=\prod_{l\in S} Z(\SL_2(\Q_l)) \rightarrow \Qbar_p^*\]
there are exactly $n:=\prod_{l\in S} (l-1)/2 $ elements $\eta \in \mathcal{X}$ such that $\eta|_{Z_S}=\mu$. 

Therefore by Proposition \ref{fibres} we have 
\[\dim_{k(x)} (\cM'_{\xbar})^{cl} = n 2^{\# S }.\]
On the other hand let $z \in Z$ be a stable point in $U$, then Proposition \ref{fibres} implies that 
\[\dim_{k(z)} (\cM'_{\zbar})^{cl} \geq n 2^{\# S + 1}.\]

As $Z$ is Zariski-dense in $U$ and $\cM'$ is coherent, the semi-continuity of the fibre rank implies that 
\[\dim_{k(x)} \cM'_{\xbar} \geq n 2^{\# S + 1}.
\] 
Therefore 
\[\cM'_{\xbar}/ (\cM_{\xbar}')^{cl} \neq 0. \]
\end{proof}

\begin{rem}\label{CMB}
\begin{enumerate}
	\item There is an analogous situation on the Coleman-Mazur eigencurve: For a modular eigenform $f \in S_{k}(\Gamma_1(N))$ of weight $k\geq 2$ with complex multiplication by $F$, the generalized eigenspace 
	\[\cM^\dagger_k(\Gamma_1(N)\cap \Gamma_0(p))_{(x_{crit})}\]
of overconvergent modular forms contains a non-classical form. For details see Proposition 2.11 and 2.13 in \cite{bellaichecrit}. Note however that it is necessary to pass to the \textit{generalized} eigenspace to see non-classical forms, the corresponding eigenspace consists only of classical forms. For the group $\SL_2$ our theorem shows that non-classical forms occur already in the eigenspace.
	\item One cannot hope for an analogue of the above theorem at the other refinement: The point $y$ is of non-critical slope and therefore $\cM^{cl}_{\overline{y}}=\cM_{\overline{y}}$. 	
	\item In the proof of the theorem above we used that there is a Zariski-dense set of stable points around the endoscopic point $x$. We want to briefly explain what happens for a point close to $x$ which comes from an endoscopic packet $\Pi$ with $p$ inert in the corresponding imaginary quadratic field. We also assume that we keep (2) to (4) in Hypothesis ($\star$) and that $\Pi_p$ is unramified. Now as $p$ is inert, the $L$-packet $\Pi_p$ is of size two and consists of the constituents of $\Ind_B^{\SL_2(\Q_p)}(\chi)$ for a quadratic character $\chi$.
Then for any choice of $\pi_f^p$ we have four elements in $\Pi$ that agree with $\pi_f^p$ at all places different from $p$ and $\infty$, and two of them are automorphic. Both automorphic representations give rise to the same point $x_{\mathrm{inert}}$ on the eigenvariety and it is non-critical. Note also that $\chi =\chi^{-1}$ and so from the proof of Lemma \ref{jacdim} we can directly see that the dimension of the classical subspace $(\cM'_{x_{\mathrm{inert}}})^{cl}$ is at least $n 2^{\#S+1}$.  
\end{enumerate}
\end{rem}

Let $f \in \cM_{\xbar}$ be a non-classical form as provided by the theorem. It would be interesting to understand the relationship between $f$ and the packet $\Pi(\theta)$ in a more concrete way.

Furthermore, using the canonical lifting $\varphi: J_B(\til{H}^1(K^p,W)_{la}) \rightarrow \til{H}^1(K^p,W)_{la}$ constructed in Section (3.4.8) of \cite{emerton-jacquet1} we get a vector $\varphi(f)$ in the $\SL_2(\Q_p)$-representation $\til{H}^1(K^p,W)_{la}$. The representation $\langle \SL_2(\Q_p)\varphi(f) \rangle$ generated by this vector is of interest in the context of the $p$-adic local Langlands programme for $\SL_2(\Q_p)$. It would be interesting to describe this representation explicitly.

\bibliography{sl2}
\bibliographystyle{plain}

\end{document}